\documentclass[12pt,a4paper]{amsart}
\usepackage{amssymb, amstext, amscd, amsmath, color}

\usepackage{url}

\usepackage[dvips]{graphicx}

\usepackage{hhline}

\usepackage{pb-diagram}
\usepackage[dvips]{graphicx}

\title[Crystal frameworks, symmetry and periodic flexes]{Crystal frameworks, symmetry and affinely periodic flexes}

\author{S.C. Power}  
\address{Department of Mathematics and Statistics, Lancaster
University, Lancaster, United Kingdom  LA1 4YF}
\email{s.power@lancaster.ac.uk}

\thanks{
Partly supported by EPSRC grant EP/J008648/1}

\theoremstyle{plain}
\newtheorem{thm}{Theorem}[section]
\newtheorem{cor}[thm]{Corollary}
\newtheorem{prop}[thm]{Proposition}
\newtheorem{lem}[thm]{Lemma}
\theoremstyle{definition}

\newtheorem{defn}[thm]{Definition}

\newcommand{\be}{\begin{equation}}
\newcommand{\ee}{\end{equation}}
\newcommand{\ber}{\begin{eqnarray}}
\newcommand{\eer}{\end{eqnarray}}
\newcommand{\barr}{\begin{array}}
\newcommand{\earr}{\end{array}}
\newcommand{\bestar}{\begin{equation*}}
\newcommand{\eestar}{\end{equation*}}

\begin{document}



%

\newcommand{\FFock}{\mathcal{F}}
\newcommand{\kil}{\mathsf{k}}
\newcommand{\Hil}{\mathsf{H}}
\newcommand{\hil}{\mathsf{h}}
\newcommand{\Kil}{\mathsf{K}}
\newcommand{\Real}{\mathbb{R}}
\newcommand{\Rplus}{\Real_+}

%

\newcommand{\bC}{{\mathbb{C}}}
\newcommand{\bD}{{\mathbb{D}}}
\newcommand{\bN}{{\mathbb{N}}}
\newcommand{\bQ}{{\mathbb{Q}}}
\newcommand{\bR}{{\mathbb{R}}}
\newcommand{\bT}{{\mathbb{T}}}
\newcommand{\bX}{{\mathbb{X}}}
\newcommand{\bZ}{{\mathbb{Z}}}
\newcommand{\bH}{{\mathbb{H}}}
%
%
\newcommand{\BH}{{\B(\H)}}
\newcommand{\bsl}{\setminus}
\newcommand{\ca}{\mathrm{C}^*}
\newcommand{\cstar}{\mathrm{C}^*}
\newcommand{\cenv}{\mathrm{C}^*_{\text{env}}}
\newcommand{\rip}{\rangle}
\newcommand{\ol}{\overline}
\newcommand{\td}{\widetilde}
\newcommand{\wh}{\widehat}
\newcommand{\sot}{\textsc{sot}}
\newcommand{\wot}{\textsc{wot}}
\newcommand{\wotclos}[1]{\ol{#1}^{\textsc{wot}}}
 \newcommand{\A}{{\mathcal{A}}}
 \newcommand{\B}{{\mathcal{B}}}
 \newcommand{\C}{{\mathcal{C}}}
 \newcommand{\D}{{\mathcal{D}}}
 \newcommand{\E}{{\mathcal{E}}}
 \newcommand{\F}{{\mathcal{F}}}
 \newcommand{\G}{{\mathcal{G}}}
\renewcommand{\H}{{\mathcal{H}}}
 \newcommand{\I}{{\mathcal{I}}}
 \newcommand{\J}{{\mathcal{J}}}
 \newcommand{\K}{{\mathcal{K}}}
\renewcommand{\L}{{\mathcal{L}}}
 \newcommand{\M}{{\mathcal{M}}}
 \newcommand{\N}{{\mathcal{N}}}
\renewcommand{\O}{{\mathcal{O}}}
\renewcommand{\P}{{\mathcal{P}}}
 \newcommand{\Q}{{\mathcal{Q}}}
 \newcommand{\R}{{\mathcal{R}}}
\renewcommand{\S}{{\mathcal{S}}}
 \newcommand{\T}{{\mathcal{T}}}
 \newcommand{\U}{{\mathcal{U}}}
 \newcommand{\V}{{\mathcal{V}}}
 \newcommand{\W}{{\mathcal{W}}}
 \newcommand{\X}{{\mathcal{X}}}
 \newcommand{\Y}{{\mathcal{Y}}}
 \newcommand{\Z}{{\mathcal{Z}}}

\newcommand{\sgn}{\operatorname{sgn}}
\newcommand{\rank}{\operatorname{rank}}
\newcommand{\corank}{\operatorname{corank}}

\newcommand{\Aut}{\operatorname{Aut}}

\newcommand{\coker}{\operatorname{coker}}

\newcommand{\Isom}{\operatorname{Isom}}

\newcommand{\nullity}{\operatorname{null}}




\date{}
\maketitle

\begin{abstract}
Symmetry equations are obtained for the rigidity matrices associated with various forms of infinitesimal flexibility
for an idealised bond-node crystal framework  $\C$ in $\bR^d$.
These equations are used to derive symmetry-adapted Maxwell-Calladine counting
formulae for periodic self-stresses and affinely periodic infinitesimal mechanisms.
The symmetry equations also lead to general Fowler-Guest  formulae
connecting the character lists of subrepresentations of the
crystallographic space and point groups which are  associated with bonds, nodes,
stresses, flexes and rigid motions.
A new derivation is also given 
for the Borcea-Streinu
rigidity matrix and the correspondence between its nullspace and the space of affinely periodic infinitesimal flexes.
\end{abstract}

\section{Introduction}

A finite bar-joint framework $(G,p)$  is a graph
$G=(V,E)$ together with a correspondence $v_i \to p_i$  between vertices and framework points
in $\bR^d$, the joints or nodes of $(G,p)$. The framework edges are the line segments $[p_i, p_j]$ associated with the edges of $G$ and these are viewed as inextensible bars. Accounts of the analysis of infinitesimal flexibility and the combinatorial rigidity of such frameworks can be found
in Asimow and Roth \cite{asi-rot}, \cite{asi-rot-2} and Graver, Servatius and Servatius \cite{gra-ser-ser}.

In the presence of a spatial symmetry group $\S$ for $(G,p)$, with isometric representation
$\rho_{sp}: \S \to {\rm Isom}(\bR^d)$,
Fowler and Guest \cite{fow-gue} considered certain unitary representations
$\rho_n\otimes \rho_{sp}$
and $\rho_{e}$ of $\S$ on the finite-dimensional vector spaces
\[
\H_v= \sum_{joints} \oplus \bR^d, \quad \H_e = \sum_{bars} \oplus \bR.
\]
These spaces contain, respectively, the linear subspace of infinitesimal flexes of $(G,p)$, denoted $\H_{\rm fl}\subseteq \H_v$, and the linear
subspace of infinitesimal self-stresses,
$\H_{\rm str}\subseteq \H_e$.
Also $\H_{\rm fl}$ contains the space $\H_{\rm rig}$ of rigid motion infinitesimal flexes while
$\H_{\rm mech}=\H_{\rm fl}\ominus \H_{\rm rig}$
is the space of infinitesimal mechanisms.
It was shown that for the induced representations
$ \rho_{\rm mech}$ and $\rho_{\rm str}$ that there is a relationship between their associated character lists.
Specifically,
\[
[{\rho}_{\rm mech}]-[{\rho}_{\rm str}]=[\rho_{sp}]\circ [{\rho}_n]-[\rho_e]-[\rho_{\rm rig}]
\]
where, for example, $[{\rho}_{\rm mech}]$ is the character list
\[
[{\rho}_{\rm mech}]=(tr({\rho}_{\rm mech}(g_1)), \dots , tr({\rho}_{\rm mech}(g_n))),
\]
for a fixed set of generators of $\S$, and where $[\rho_{sp}]\circ [{\rho}_n]$ indicates entry-wise product.

The Fowler-Guest formula can be viewed as a symmetry-adapted version of Maxwell's counting rule for finite rigid frameworks. Indeed, evaluating the formula
at the identity element $g_1$ one recovers the more general Maxwell-Calladine equation, which in three dimensions takes the form
\[
m-s= 3|V|-|E|-6,
\]
where $m$ and $s$ are the dimensions of the spaces of infinitesimal mechanisms and infinitesimal self-stresses, respectively.
Moreover it has been shown in a variety of studies that the character equation for individual symmetries can lead to useful symmetry-adapted counting conditions for rigidity and isostaticity (stress-free rigidity). See, for example, Connelly, Guest, Fowler, Schulze and Whiteley \cite{con-et-al}, Fowler and Guest \cite{fow-gue-2}, Owen and Power \cite{owe-pow-FSR} and Schulze \cite{sch-finflex}, \cite{sch-block}.

The notation above is taken from  Owen and Power \cite{owe-pow-FSR} where symmetry equations were given for the rigidity matrix $R(G,p)$
and a direct proof of the character  equation obtained, together with a variety of
applications.  In particular
a variant of the character equation
was given for translationally  periodic infinite frameworks and for strict  periodicity.
In the present development, which is self-contained, we combine the symmetry equation approach with a new derivation of the rigidity matrix, identified in Borcea and Streinu \cite{bor-str},
which is associated with affinely periodic infinitesimal flexes. This derivation is given
via infinite rigidity matrices and this perspective is well-positioned for the
incorporation of space group symmetries. By involving such symmetries we obtain
general symmetry-adapted affine Maxwell-Calladine equations for periodic frameworks  (see Theorem \ref{t:MC} for example) and affine variants of the  character list formula.

To be more precise,
let $\C$ be a countably infinite bar-joint framework in $\bR^d$ with discrete vertex set and with translational periodicity for a full rank subgroup $\T$ of isometric translations.
We refer to frameworks with this form of periodicity as \textit{crystal frameworks}.
An affinely periodic flex of $\C$ is one that, roughly speaking, allows periodicity relative
to an affine transformation of the ambient space, including contractions, global rotations
and sheering. The continuous and infinitesimal forms  of this are given in Definition \ref{d:Flexes} and it is shown how  the  infinitesimal  flexes of this type correspond to  vectors in the null space of a finite rigidity matrix $R(\M,\bR^{d^2})$ with
$|F_e|$ rows and $d|F_v|+d^2$ columns. Here $\M=(F_v,F_e)$, a \textit{motif} for $\C$,
is a choice of a set $F_v$ of vertices
and a set $F_e$ of edges whose translates partition the vertices and edges of $\C$.
The "extra" $d^2$ columns correspond to the $d^2$ degrees of freedom present in an affine adjustment of the given periodicity.
When the affine adjustment matrices are restricted  to a subspace $\E$ there is an associated
rigidity matrix and linear transformation, denoted $R(\M,\E)$.
The case of strict periodicity corresponds to $\E=\{0\}$.

The symmetry-adapted Maxwell-Calladine formulae
for affinely periodic flexibility
are derived from symmetry equations of independent interest which take the form
\[
\label{e:SymmetryEquation}
 \pi_e( g)R(\M,\E)=R(\M,\E) \pi_v(g),
\]
where $\pi_e$ and $\pi_v$ are finite-dimensional representations of the space group
$\G(\C)$ of $\C$. As we discuss in Section 3, these in turn derive from
more evident symmetry equations for the infinite rigidity matrix $R(\C)$.

The infinite-dimensional vector space transformation perspective for
the infinite matrix $ R(\C)$ is also useful in the consideration of  quite general infinitesimal flexes
including local flexes, diminishing flexes, supercell-periodic flexes and, more generally, phase-periodic flexes. In particular phase-periodic infinitesimal flexes for  crystal
frameworks  are  considered in Owen and Power \cite{owe-pow-crystal} and Power \cite{pow-matrix}
in connection with the analysis of  rigid unit modes (RUMs) in material crystals, as described in Dove et al \cite{dov-exotic} and Wegner \cite{weg}, for example. See also Badri, Kitson and Power \cite{bad-kit-pow}.

In the final section we give some illustrative examples.  The well-known kagome framework and a framework suggested by a Roman tiling of the plane, are contrasting examples which are both in Maxwell counting equilibrium. We also consider some overconstrained  frameworks,
including $\C_{\rm Hex}$  which is formed as a regular network of
triangular faced hexahedra, or bipyramids. In addition to such curious mathematical frameworks we note that there exist a profusion of  polyhedral net frameworks which are suggested by the crystalline structure of natural materials, such as quartz, perovskite, and various aluminosilicates and zeolites. Some of these frameworks, such as the cubic sodalite framework,  have pure mathematical forms. Examples of this type can be found in
Borcea and Streinu \cite{bor-str}, Dove et al \cite{dov-exotic},
Kapko et al \cite{kap-daw-tre-tho}, Power \cite{pow-matrix} and  Wegner  \cite{weg} for example.

A number of recent articles also consider affinely periodic flexes and rigidity.
Ross, Schulze and Whiteley \cite{ros-sch-whi}
consider Borcea-Streinu style rigidity matrices and orbit rigidity matrices to obtain a counting predictor for infinitesimal motion  in the presence  of  space symmetries of separable type.
(See also Corollary \ref{c:IneqPredictor} below.)
Using this they  derive the many counting conditions  in two and three dimensions
 which arise from  the type of the separable space group in combination  with the type of the affine axial velocity space $\E$.
In a more combinatorial vein Borcea and Streinu \cite{bor-str-2} and  Malestein and Theran \cite{mal-the}
consider a detailed combinatorial rigidity theory for periodic frameworks.
In particular they obtain combinatorial characterisations of periodic affine infinitesimal rigidity and isostaticity for generic periodic frameworks in $\bR^2$.

The author is grateful for discussions with Ciprian Borcea, Tony Nixon and John Owen.


\section{Crystal frameworks and rigidity matrices.}

\subsection{Preliminaries.}
For convenience suppose first that $d$ is equal to $3$. The changes needed for the extension to $d=2,4,5,\dots $ are essentially notational.
Let $G=(V,E)$ be a countable simple graph with $V=\{v_1,v_2,\dots \}, E \subseteq V \times V$
and let $p$ be a sequence $(p_i)$ where the $p_i= p(v_i)$  are corresponding  points in $\bR^3$. The pair $(G,p)$ is  said to be a \textit{bar-joint framework} in $\bR^3$ with framework vertices $p_i$ (the joints or nodes)
and framework edges (or bars) given by the straight line segments $[p_i, p_j]$
between $p_i$ and $p_j$ when $(v_i, v_j)$ is an edge in $E$.
We assume that the bars have positive lengths.

An \textit{infinitesimal flex} of $(G,p)$ is a vector $u=(u_i)$, with each $u_i \in\bR^3$ regarded
as a velocity applied to $p_i$, such that for each edge $[p_i, p_j]$
\[
\langle p_i-p_j, u_i\rangle = \langle p_i-p_j, u_j\rangle.
\]
This simply asserts that for each edge the components in the edge direction of the endpoint velocities are in agreement. Equivalently, an infinitesimal flex is a velocity vector $v=(v_i)$ for which the distance deviation
\[
|p_i-p_j|-|(p_i+tu_i)-(p_j+tu_j)|
\]
of each edge is of order $t^2$ as $t\to 0$.

An \textit{isometry} of $\bR^3$ is a distance-preserving map $T:\bR^3\to \bR^3$. A
\textit{full rank translation group} $\T$ is a set of translation isometries
$\{T_k:k\in \bZ^3\}$ with  $T_{k+l}=T_k+T_l
$ for all $k,l$, $T_k \neq I$ if $k\neq 0$,
such that the three vectors
\[
a_1=T_{e_1}0, \quad a_2=T_{e_2}0,\quad  a_3 = T_{e_3}0,
\]
associated with the generators $e_1=(1,0,0), e_2=(0,1,0), e_3=(0,0,1)$
of $\bZ^3$ are not coplanar. We refer to these vectors as the \textit{period vectors} for
$\T$ and $\C$. The following definitions follow the formalism of  Owen and Power \cite{owe-pow-crystal},
\cite{pow-matrix}.

\begin{defn}
A crystal framework $\C=(F_v, F_e, \T)$ in $\bR^3$, with full rank translation group $\T$
and finite {motif} $(F_v, F_e)$, is a countable bar-joint framework with framework points $ p_{\kappa,k}$,
for $1\leq \kappa \leq t,  k\in \bZ^3$, such that

(i) $F_v=\{p_{\kappa,0}:1\leq \kappa \leq t\}\}$ and $F_e$ is a finite set of framework edges,

(ii) for each $\kappa$ and $k$ the point $p_{\kappa,k}$ is the translate $T_kp_{\kappa,0}$,

(iii) the set $\C_v$ of framework points is the disjoint union of the sets
$T_k(F_v), k\in \bZ^3$,

(iv) the set $\C_e$ of framework edges is the disjoint union of the sets
$T_k(F_e), k\in \bZ^3$.
\end{defn}

The finiteness of the motif and the full rank of $\T$ ensure that the periodic set $\C_v$ is a discrete set in $\bR^d$ in the usual sense.

One might also view the motif as a choice of representatives for the translation equivalence classes of the vertices and the edges of $\C$. It is natural to take $F_v$ as the vertices of $\C$ that lie in a distinguished \textit{unit cell},
such as $[0,1)\times[0,1)\times[0,1)$ in the case that $\T$ is the
standard cubic translation group for $\bZ^3 \subseteq \bR^3$.  In this case one could take $F_e$ to consist of the edges lying in the unit cell together with some number
of cell-spanning edges $e=[p_{\kappa,0}, p_{\tau, \delta}]$ where
$\delta =\delta(e) = (\delta_1,\delta_2,\delta_3)$ is nonzero.
This can be a convenient assumption and with it we may refer to
$\delta (e)$ as the \textit{exponent} of the edge.  The case of general motifs is discussed
at the end of this section.

While we focus on discrete full rank periodic bar-joint frameworks, which appear in many
mathematical models in applications, we remark that
there are, of course, other interesting forms of infinite frameworks with infinite spatial symmetry groups. In particular there is the class of \textit{cylinder frameworks} in
$\bR^d$  which are translationally periodic for a subgroup $\{T_k: k\in \bZ^r\}$ of a full rank translation group $\T$. See, for example, the two-dimensional strip frameworks of \cite{owe-pow-crystal} and   the hexahedral tower in Section 5. On the other hand one can also consider infinite frameworks constrained to infinite surfaces and employ restricted
framework rigidity theory \cite{nix-owe-pow}.

\subsection{Affinely periodic infinitesimal flexes.}
We shall give various definitions of affinely periodic infinitesimal flexes and we start with 
a finite  matrix-data description and give a connection with certain continuous flexes.
This notation takes the form $(u, A)$ where $A$ is an arbitrary
$d \times d$ matrix and $u=(u_{\kappa})_{\kappa\in F_v}$ is a real vector in $\bR^{d|F_v|}$
composed of the infinitesimal flex  vectors $u_\kappa=u_{\kappa,0}$ that are assigned to the framework vertices $p_{\kappa,0}$. The matrix
$A$, an \textit{affine velocity matrix}, is viewed both as a $d \times d$ real matrix
and as a vector in $\bR^{d^2}$ in which the  columns of $A$ are written in order.
Thus, each pair $(v, B)$  generates an  assignment of
displacement velocities to the framework points, and 
in some cases the resulting velocity vector will qualify as an
affinely periodic infinitesimal flex  in the sense given below.
In Theorem \ref{t:FlexCompare} we obtain a three-fold characterisation of such flexes.

By an \textit{affine flow} we mean a differentiable function $t \to A_t$ from $[0, t_1 ]$ to the set of nonsingular linear transformations of $\bR^3$ such that $A_0=I$ and we write $A$ for the derivative  at
$t=0$. (See also Owen and Power \cite{owe-pow-crystal}.) In particular for any matrix $A$ the function $t \to I+tA$ is an affine flow, for suitably small $t_1$.

\begin{defn}\label{d:Flexes}
Let $\C$ be a crystal framework.

(i) A flow-periodic flex of $\C$  for an affine flow $t \to A_t$ is a coordinate-wise differentiable path
$p(t)= (p_{\kappa,k}(t))$, for $t\in[0, t_1 ]$ for some $t_1>0$, such that

(a) the flow-periodicity condition holds;
\[
p_{\kappa, k}(t) = A_tT_kA_t^{-1}p_{\kappa, 0}(t),\quad \mbox{for all  }\quad  \kappa, k, t,
\]

(b) for each edge
$e=[p_{\kappa,k}, p_{\tau,k+\delta(e)}]$  the distance function
\[
t\to d_e(t) :=|p_{\kappa,k}(t)-p_{\tau,k+\delta(e)}(t)|
\]
is constant.

(ii)
An affinely periodic infinitesimal flex of $\C$
is a pair $(u,A)$, with $u\in \bR^{3|F_v|}$ and $A$ a $d\times d$ real matrix such that for all motif edges $e=[p_{\kappa,0},p_{\tau,\delta(e)}]$ in $F_e$,
\[
\langle p_{\kappa,0}- p_{\tau,\delta(e)},u_{\kappa}- u_{\tau}+A(p_{\tau,0}- p_{\tau,\delta(e)}) \rangle =0.
\]
\end{defn}

The notion of a flow-periodic flex, with finite path vertex motions, is quite intuitive and easily illustrated. One can imagine for example a periodic zig-zag bar framework (perhaps featuring as a subframework of a full rank  framework) which flexes periodically by concertina-like contraction and expansion.
Affinely periodic infinitesimal flexes are perhaps less intuitive. However we have the following proposition.

\begin{prop}\label{p:FlexDerivative} Let $p(t)$ be a flow-periodic flex of $\C$ for the flow $t \to A_t$
whose derivative at $t=0$ is $A$.
Then $(p'_{\kappa,0}(0),A)$ is an affinely periodic infinitesimal flex of $\C$.
\end{prop}

\begin{proof} Differentiating the  flow-periodicity condition
$$A_t^{-1}p_{\tau,\delta}(t)=T_{\delta}A_t^{-1}p_{\tau, 0}(t)$$
 gives
\[
-Ap_{\tau,\delta}(0)+p_{\tau,\delta}'(0)= -Ap_{\tau,0}(0)+p_{\tau,0}'(0)
\]
and so
\[
p'_{\tau,\delta}(0)=p'_{\tau,0}(0)+A(p_{\tau,\delta}(0)-p_{\tau,0}(0)).
\]
Differentiating and evaluating at zero the constant function
\[
t \to \langle p_{\kappa,0}(t)-p_{\tau,\delta(e)}(t), p_{\kappa,0}(t)-p_{\tau,\delta(e)}(t)\rangle,
\]
for the edge $e=[p_{\kappa,0},p_{\tau,\delta(e)}]$ in $F_e$, gives
\[
0=\langle p_{\kappa,0}(0)-p_{\tau,\delta(e)}(0),p'_{\kappa,0}(0)-
p'_{\tau,\delta(e)}(0) \rangle.
\]
Thus, substituting, with $\delta(e)= \delta$, gives
\[
0=\langle p_{\kappa,0}(0)-p_{\tau,\delta(e)}(0),p'_{\kappa,0}(0)-
p'_{\tau,0}(0)+A(p_{\tau,0}(0)-p_{\tau,\delta(e)}(0)) \rangle,
\]
as required.
\end{proof}

We remark that with a similar proof one can show the following  related
equivalence. The pair $(u,A)$ is an affinely periodic
infinitesimal flex of $\C$ if and only if for the affine flow $A_t= I+tA$
the distance deviation
\[
|p_{\kappa,0}-p_{\tau,\delta(e)}|-|(p_{\kappa,0}+tu_{\kappa})-
A_tT_{\delta(e)}A_t^{-1}(p_{\tau,\delta(e)}+tu_{\tau})|
\]
of each edge is of order $t^2$ as $t\to 0$. This property
justifies, to some extent, the terminology that $(u, A)$ is an infinitesimal flex for $\C$. Note that an infinitesimal rotation flex of $\C$ (defined naturally, or by means of Proposition
\ref{p:FlexDerivative}) provides an affine infinitesimal flex.


\subsection{Rigidity matrices.}
We first  define a  finite matrix  associated with
the motif $\M=(F_v, F_e)$ which is essentially the rigidity matrix identified by Borcea and Streinu
\cite{bor-str}. We  write it as $R(\M,\bR^{d^2})$ and also view it as a linear transformation
from $\bR^{d|F_v|}\oplus \bR^{d^2}$ to $\bR^{|F_e|}$.
In the case $d=3$ the summand space
$ \bR^{d^2}$ is a threefold direct sum
$\bR^3\oplus \bR^3 \oplus \bR^3$. For
$\C$ with cubic lattice structure with period vectors
\[
a_1=(1,0,0),a_2=(0,1,0), a_3 =(0,0,1),
\]
an \textit{affine velocity matrix} $A=(a_{ij})$
provides a row vector in  $ \bR^{9}$ in which the columns of $A$ are
written in order.
It is shown in Theorem \ref{t:FlexCompare} that the composite vector $(u,A)$ lies in the kernel of this rigidity matrix  if and only if $(u,A)$  is an affinely periodic infinitesimal
flex in the sense of Definition \ref{d:Flexes} (ii), and if and only if the related infinite
vector $\tilde{u}= (u_\kappa-Ak)_{\kappa,k}$ is an infinitesimal flex of $\C$ in the usual bar-joint framework sense.

\begin{defn}\label{d:affrigmatrix}
Let $\C=(F_v, F_e, \T)$ be a crystal framework in $\bR^d$ with isometric translation group
$\T$
and motif $\M=(F_v, F_e)$. Then the affinely periodic rigidity
matrix $R(\M,\bR^{d^2})$ is the $|F_e| \times (d|F_v|+d^2)$ real matrix whose rows, labelled by
the edges $e=[p_{\kappa,0},p_{\tau,\delta(e)}]$ of $F_e$, have the form
\[
[0\cdots 0~~ v_e~~ 0\cdots 0 ~~-v_e~~ 0~~\cdots ~~0 ~~\delta_1v_e \cdots \delta_dv_e],
\]
where $v_e=p_{\kappa,0}-p_{\tau,\delta(e)}$ is the edge vector for $e$, distributed in the
$d$ columns for $\kappa$, where $-v_e$ appears in the columns for $\tau$, and
where $\delta(e)=(\delta_1,\dots ,\delta_d)$ is the exponent of $e$. If $e$ is a reflexive edge in the sense that $\kappa = \tau$ then the entries in the $d$ columns for $\kappa$ are zero.
\end{defn}

The transformation $R(\M,\bR^{d^2})$ has the block form
$[R(\M)~~X(\M)]$ where  $R(\M)$
is the (strictly) periodic rigidity matrix, or motif rigidity matrix.
We also define the  linear transformations $R(\M,\E)$ for linear subspaces  $\E\subseteq \bR^{d^2}$  of
affine velocity  matrices,  by restricting the domain;
\[
R(\M,\E): \bR^{3|F_v|}\oplus \E \to \bR^{|F_e|}.
\]
For example in three dimensions one could take $\E$ to be the three dimensional space of diagonal matrices and this would provide the rigidity matrix which detects the persistence of infinitesimal rigidity even in the presence of axial expansion and contraction. (See also \cite{ros-sch-whi} for other interesting forms of relaxation of strict periodicity.)

For a general countably infinite bar-joint framework $(G,p)$ one may define a rigidity matrix
$R(G,p)$ in the same way as  for finite bar-joint frameworks. For the crystal frameworks $\C$ it takes the following form.

\begin{defn}\label{d:InfiniteRigidityMatrix}
Let $\C=(F_v, F_e, \T)$ be a crystal framework in $\bR^d$ as given in Definition \ref{d:affrigmatrix}. Then the  infinite rigidity
matrix for $\C$ is the real matrix  $R(\C)$
whose rows, labelled by
the edges $e=[p_{\kappa,k},p_{\tau ,k+\delta(e)}]$ of $\C$, for $k\in\bZ^d$, have the form
\[
[\cdots 0\cdots 0~~ v_e~~ 0\cdots 0 ~~-v_e~~ 0~~\cdots ~~0\cdots ],
\]
where $v_e=p_{\kappa, k}-p_{\tau ,k+\delta(e)}$ is the edge vector for $e$, distributed in the $d$ columns for $\kappa, k$ and where $-v_e$ appears in the columns for $(\tau, k+\delta (e))$, and
where $\delta(e)=(\delta_1,\dots ,\delta_d)$ is the exponent of $e$.
\end{defn}

We view $R(\C)$ as a linear transformation from $\H_v$ to $\H_e$
 where $\H_v$ and $\H_e$ are the direct product vector spaces,
\[
\H_v= \Pi_{\kappa,k}\bR^3, \quad \H_{e,k}= \Pi_{k}\bR.
\]
In particular, as in the finite framework case, a velocity vector $v$ in $\H_v$ is an infinitesimal flex of $\C$ if and only if $v$ lies in the nullspace of $R(\C)$.

The following theorem shows the role played by the rigidity matrices $R(\M,\bR^{d^2})$ and
$R(\C)$ in locating affinely periodic infinitesimal flexes. For the general case we require the invertible transformation
 $Z:\bR^d \to \bR^d$ which maps the standard basis vectors
$\gamma_1,  \dots ,\gamma_d$ for $\bR^d$ to the period vectors $a_1, \dots ,a_d$ of $\C$.
In the next definition a finite vector-matrix pair $(v,A)$ generates a velocity vector in $\H_v$
of affinely periodic type.

\begin{defn} \label{d:AffineDisplacements} The space $\H_v^{\rm aff}$ is the vector subspace of $\H_v$
consisting of \textit{affinely periodic velocity vectors}
$\tilde{v}=(\tilde{v}_{\kappa,k})$, each of which is determined by a finite vector $v=(v_\kappa)_{\kappa\in F_v}$
in $\bR^{|F_v|}$ and a $d \times d$ real matrix $A$ by the equations
\[
\tilde{v}_{\kappa,k}=v_\kappa -AZk, \quad k\in \bZ^d.
\]
\end{defn}
\medskip

Note that $\H_v^{\rm aff}$ is a linear subspace of $\H_v$. For if $\tilde{u}, \tilde{v}$ correspond to $(u,A)$ and $(v,B)$ respectively, then
\[
\tilde{u}_{\kappa,k}+ \tilde{v}_{\kappa,k}=u_\kappa-AZk+v_\kappa-BZk
\]
\[
=(u_\kappa+v_\kappa) -(A+B)Zk
\]
which defines the displacement velocity corresponding to $(u+v,A+B)$.

In Definition \ref{d:Flexes} (ii) an infinitesimal flex is denoted by a \emph{vector-matrix pair} $(u,A)$
with $u\in \bR^{d|F_v|}$, corresponding to vertex displacement velocities in a unit cell,
and a $d \times d$ matrix $A$ corresponding to axis displacement velocities.
In the next theorem 
we identify this  vector-matrix data $(u,A)$ with\emph{vector-vector} data $(u,AZ)$
in which $AZ$ denotes
a  vector of length $d^2$ given by the $d$ vectors
$AZ\gamma_1, \dots , AZ\gamma_d$.

\begin{thm}\label{t:FlexCompare} Let $\C$ be a crystal framework in $\bR^d$
with translation group
$\T$ and period vector matrix $Z$. Then the restriction of the rigidity
matrix transformation $R(\C):\H_v \to \H_e$ to the finite-dimensional space
$\H_v^{\rm aff}$ has representing matrix $R(\M,\bR^{d^2})$. Moreover,
the following statements are equivalent:

(i) The vector-matrix pair $(u,A)$, with $u\in \bR^{d|F_v|}$ and $A\in M_d(\bR)$, is an affinely periodic infinitesimal flex for $\C$.

(ii) The vector-vector pair $(u, AZ)$ lies in the nullspace of $R(\M,\bR^{d^2})$.

(iii) The vector $\tilde{u}$ lies in the nullspace of $R(\C)$, where $\tilde{u}\in \H_v$ is the vector
defined by the affinely periodic  extension formula
\[
\tilde{u}_{\kappa,k}=u_\kappa -AZk,\quad k\in \bZ^d.
\]
\end{thm}

\begin{proof}
Let $e=[p_{\kappa,0},p_{\tau,\delta(e)}]$ be an edge in $F_e$ and as before write
$v_e$ for the edge vector $p_{\kappa,0}-p_{\tau,\delta(e)}$. Note that
the term $A(p_{\tau}- p_{\tau,\delta (e)})$ in Definition \ref{d:Flexes}
is equal to
$A(\delta_1a_1 + \dots + \delta_da_d) = AZ(\delta_1\gamma_1+\dots +\delta_d\gamma_d)$.
Let $\eta_e, e\in F_e,$ be the standard basis vectors for $\bR^{|F_e|}.$
The inner product in Definition \ref{d:Flexes} may be written

\[
\langle p_{\kappa,0}- p_{\tau,\delta(e)},u_{\kappa}- u_{\tau}+A(p_{\tau,0}- p_{\tau,\delta(e)}) \rangle
\]
\begin{eqnarray*}
&=&\langle v_e,u_{\kappa}- u_{\tau}\rangle + \langle v_e, \sum_i AZ(\delta_i\gamma_i)\rangle\\
&=&
\langle R(\M)u,\eta_e\rangle +\sum_i\langle\delta_iv_e,AZ\gamma_i\rangle\\
&=&\langle R(\M)u,\eta_e\rangle +\langle(X(\M)(AZ),\eta_e\rangle\\
&=&\langle R(\M,\bR^{d^2})(u,AZ),\eta_e\rangle.
\end{eqnarray*}
Thus the equivalence of (i) and (ii) follows.

That the statements (i) and (ii) are equivalent to (iii) follows from the following calculation.
Let $w=(w_\kappa)$ be a vector in $\bR^{d|F_v|}$, let $(w,B)$ be vertex-matrix data and let $\tilde{w}$ be the associated velocity vector in $\H_v^{\rm aff}$. Then, for $e$ in $F_e$, given by
$[p_{\kappa,0},p_{\tau,\delta(e)}]$,  the
$e,k$ coordinate of the vector $R(\C)\tilde{w}$ in $\H_e$ is given by

\begin{eqnarray*}
(R(\C)\tilde{w})_{e,k}   &=&\langle v_e, \tilde{w}_{\kappa, k}\rangle - \langle v_e, \tilde{w}_{\tau, k+\delta(e)}\rangle
\\
&=&\langle v_e, \tilde{w}_{\kappa, k} - \tilde{w}_{\tau, k}\rangle +\langle v_e, BZ\delta (e) \rangle\\
&=& \langle v_e, w_{\kappa} - w_{\tau}\rangle +\sum_i \langle \delta_iv_e, BZ\gamma_i \rangle
\\
&=&(R(\M)w)_e + (X(\M)(BZ))_e\\
&=&(R(\M,\bR^{d^2})(w,BZ))_e.
\end{eqnarray*}
\end{proof}

We say that a periodic bar-joint framework $\C$ is \textit{affinely periodically infinitesimally rigid} if the only  affinely periodic infinitesimal flexes are those in the
finite-dimensional space $\H_{\rm rig}\subseteq \H_v$ of infinitesimal flexes induced by isometries of $\bR^d$.
The following theorem was obtained by Borcea and Streinu \cite{bor-str} where the rigidity matrix
was derived from a projective variety identification of the
finite flexing space (configuration space) of the framework.

\begin{thm}\label{t:BCRigidityMatrixThm}
A periodic bar-joint framework $\C$ with motif $\M$ is affinely periodically infinitesimally rigid if and only if
\[
\rank R(\M,\bR^{d^2}) = d|F_v|+ d(d-1)/2.
\]
\end{thm}

\begin{proof}
In view of the description
of affinely periodic flexes in Theorem \ref{t:FlexCompare} the rigidity requirement is equivalent to
the equality
$$\rank R(\M,\bR^{d^2}) = d|F_v|+d^2 -\dim \H_{rig}.$$
The space $\H_{\rm rig}$
has dimension $d(d+1)/2$, since a basis may be provided by $d$ infinitesimal translations and $d(d-1)/2$ independent infinitesimal rotations,
and so the proof is complete.
\end{proof}

The theorem above is generalised by   the Maxwell-Calladine formula in Theorem \ref{t:M-Caffine} which incorporates the following space of self-stresses for  $\C$ relative to $\E$.
Symmetry adapted variants are given in Corollary \ref{c:FG} and Theorem \ref{t:MC}.

\begin{defn}\label{d:selfstress}
An \textit{infinitesimal self-stress} of a countable bar-joint framework $(G,p)$, with vertices of finite degree,
is  a vector in $\H_e$ lying in the cokernel of the infinite rigidity matrix $R(G,p)$.
\end{defn}

For a periodic framework $\C$ such vectors  $w=(w_{e,k})$ are characterised, as in the case of finite frameworks,
by a linear dependence row condition, namely
\[
\sum_{e,k} w_{e,k}R(\C)_{((e,k),(\kappa,\sigma,l))} = 0, \quad \mbox{ for all } \kappa,\sigma,l.
\]
Alternatively this can be paraphrased in terms of local conditions
\[
\sum_{(\tau, l): [p_{\kappa,k},p_{\tau,l}]\in \C_e} w_{\tau, l}(p_{\kappa,k}-p_{\tau,l})=0,
\]
which may be interpreted as a balance of internal stresses.

From the calculation in Theorem \ref{t:FlexCompare} it follows that the rigidity matrix $R(\C)$ maps affinely periodic displacement velocities  to periodic vectors in $\H_e$. Write $\H_e^{\rm per}$ for the vector space of such vectors and
$\H^{\rm aff}_{\rm str}$ (resp. $\H^\E_{\rm str}$) for the subspace of periodic self-stresses that correspond to vectors in the cokernel of $R(\M,\bR^{d^2}$) (resp. $R(\M,\E$), and
let $s_\E$ denote the vector space dimension of $\H^\E_{\rm str}$,
Also let  $f_\E$ denote the dimension of  the space $\H_{\rm rig}^\E$ of infinitesimal affinely periodic rigid motions with data $(u,A)$ with  affine velocity matrix $A\in \E$,
and  let $m_\E$ denote the dimension of the space of infinitesimal mechanisms
of the same form. Thus $m_\E = \dim \H^\E_{\rm mech}$ where $\H^\E_{\rm fl}=\H^\E_{\rm mech}\oplus \H_{\rm rig}^\E$.

For $\E=\bR^{d^2}$ the following theorem is Proposition 3.11 of \cite{bor-str}.

\begin{thm}\label{t:M-Caffine}
Let $\C$ be a crystal framework in $\bR^d$ with given translational periodicity and let
$\E\subset M_d(\bR)$ be a linear space of affine velocity
matrices. Then
\[
m_\E-s_\E=d|F_v|+\dim \E -|F_e| - f_\E
\]
where $m_\E$ is the dimension of the space of  non rigid motion
$\E$-affinely periodic infinitesimal flexes and where $s_\E$ is the dimension of the
space of periodic self-stresses for  $\C$ and $\E$.
\end{thm}

\begin{proof}
By Theorem \ref{t:FlexCompare} we may identify $\H^{\rm aff}_v$ with  $\bR^{d|F_v|}\oplus \E$ and we may identify $\H^{\E}_v$  with $\bR^{d|F_v|}\oplus \bR^{d^2}$, the domain
space of $R(\M,\E)$.
We have
\[
\bR^{d|F_v|}\oplus \E=(\H_v^\E \ominus \H^\E_{\rm mech})\oplus \H_{\rm mech}^\E \oplus
 \H_{\rm rig}^\E
\]
and
\[
\bR^{|F_e|}= (\H_e^{\rm per}\ominus \H^{\E}_{\rm str})\oplus \H^{\E}_{\rm str}.
\]
With respect to this decomposition  $R(\M,\E)$
takes the block form
$$R(\M,\E)=\begin{bmatrix}
R&0&0\\0& 0&0\end{bmatrix}.
$$
with $R$ an invertible matrix. Since $R$ is square it follows that
\[
d|F_v| +\dim\E -(m_\E+f_\E)=|F_e|-s_\E,
\]
as required.
\end{proof}

We remark that one may view an affinely periodic infinitesimal flex $(u, A)$  as an infinitesimal flex of a  finite framework located in the  set $\bR^3/ \T$
regarded as a  distortable torus. The edges of such a framework are 
determined by the motif and include "wrap-around" edges  determined naturally by the edges with nontrivial exponent.  Here, in effect,  the opposing (parallel) faces of the unit cell parallelepiped are identified, so that one can define (face-to-face) periodicity in the natural way as the corresponding concept for translation periodicity.  One can specify flex periodicity modulo an affine velocity impetus of the torus, given by the matrix $A$.  
 See also Whiteley \cite{whi-union}.

\subsection{The general form of $R(\M,\bR^{d^2})$.}
We have so far made the notationally simplifying assumption that the motif
set $F_e$ consists of edges with at least one vertex in the set $F_v$.
One can always choose such motifs and the notion of the exponent of such edges is natural.
However, more general motifs are natural should one wish to highlight polyhedral units or symmetry and
we now note the minor adjustments needed for the general case. See also \cite{pow-matrix}.

Let $e=[p_{\kappa,k},p_{\tau,k'}]$ be an edge of $F_e$ with both $k$ and $k'$ not
equal to the zero multi-index and suppose first that $\kappa \neq \tau$. Then the column entries in row $e$ for the $\kappa$ label are, as before, the entries of the edge vector
$v_e=p_{\kappa,k}-p_{\tau,k'} $, and the entries in the $\tau$ columns are, as before, those of $-v_e$. The final $d^2$ entries are modified using the generalised edge exponent $\delta(e) = k'-k$. Note that this row is determined up to sign by the vertex order for the framework edge and this sign could be fixed by imposing an order on the edge. (See also the discussion in Section 4 of the various labelled graphs.) If $e$ is a reflexive edge in the sense that $\kappa = \tau $ then once again the entry for the $\kappa$ labelled columns are zero and the final $d^2$ entries
similarly use the generalised exponent.

\section{Symmetry equations and counting formulae.}

\subsection{Finite-dimensional representations of the space group.}
We first obtain symmetry equations for the affine rigidity matrix
$R(\M,\bR^{d^2})$ with respect to  representations of the space group of $\C$.

Let $\G(\C)$ be the abstract crystallographic group of the crystal framework
$\C=(F_v,F_e, \T)$. This is  the space group
of isometric maps of $\bR^d$ which map $\C$ to itself, viewed
as an abstract group. This entails  the following
two assertions.
\medskip

(i) Each symmetry element $g$ in $\G(\C)$ acts as a permutation of the framework vertex labels,
\[
(\kappa, k) \to g\cdot(\kappa, k),
\]
and this permutation induces a  permutation  of the framework edge labels,
\[
(e,k)\to  g\cdot(e, k).
\]

(ii) There is a representation $\rho_{sp}: g\to T_g$ of $\G(\C)$ in  $\Isom(\bR^d)$, the spatial representation, which extends the indexing  map $\bZ^d \to \T$
and is  such that for each framework vertex
\[
p_{g\cdot({\kappa, k})}= T_g(p_{\kappa,k}).
\]
\medskip

The permutation action on vertex labels simplifies when the symmetry $g$ is \emph{separable} with respect to $\T$. By this we mean that 
$g\cdot(\kappa, k)=(g\cdot \kappa, g\cdot k)$ where  $\kappa \to g\cdot\kappa$
and $k\to g \cdot k$ are group actions on $F_v$ and $\T$ for the group generated by $g$. This may occur, for example, for a translation group with orthonormal period vectors
and symmetries of axial reflection, or rotation by $\pi/2$ or $\pi$.
A simplification also occurs for a weaker property, namely when the symmetry $g$ is \emph{semiseparable} relative to the translation group (or motif). This requires
only that there is an induced action $\kappa \to g\cdot\kappa$ on the vertex set of the motif, so that
\[
 g\cdot(\kappa, k)=(g\cdot\kappa, k')
\]
where $k'$ generally depends on $g, \kappa$ and $k$. This is the case, for example, for the
rotational symmetries of the kagome framework and the Roman tiling framework given in Section 4. The group $\T$ is taken to be the natural (maximal) translational symmetry group. In the case of the Roman tile framework and rotation by $\pi/6$ the action on $F_v$ is free, whereas for this rotation symmetry and the kagome framework the action is not free.

In general a finite order symmetry $g$ of $\C$ need not induce a permutation action on $F_v$. One need only consider rotation symmetries for the basic grid framework whose framework vertex set is ${\bZ^2}$ together with the nonstandard translation group with vertex motif $F_v= \{(0,0),(0,1)\}$.


Assume that the group action of $\G(\C)$ on $F_v$ is semiseparable.
On vertex coordinate  labels
define the corresponding action
$g\cdot (\kappa,\sigma,k)=(g\cdot \kappa,\sigma,k')$, where $\sigma$, in case $d=3$, denotes $x, y$
or $z$.
Let $\rho_v$ be the permutation representation of $\G(\C)$ on the  velocity space $\H_v$, as (linear) vector space transformations,  which is defined by
\[
(\rho_v(g)u)_{\kappa, \sigma, k}=u_{g^{-1}\cdot(\kappa, \sigma, k)}
 \]
Similarly let $\rho_e$ be the representation of $\G(\C)$ on $\H_e$
such that
\[
(\rho_e(g)w)_{e,k}= w _{g^{-1}\cdot(e, k)}.
\]
These linear transformations may also be defined as  forward shifts
in the sense that, for $l\in \bZ^d$,
\[
\rho_v(l): \xi_{\kappa,\sigma,k} \to \xi_{\kappa,\sigma,k+l},
\]
\[
\rho_e(l): \eta_{e,k} \to \eta_{e,k+l}.
\]
where  $\{\xi_{\kappa,\sigma,k}\}$,   $\{\eta_{e,k}\}$ are natural coordinate "bases".
Here $\eta_{e,k}$ denotes the element of $\H_e$, given in terms of the Kronecker symbol, by
\[
(\eta_{e,k})_{e',k'} = \delta_{e,e'}\delta_{k,k'}.
\]
The totality of these vectors gives a normal vector space basis for the set of finitely nonzero vectors of $\H_v$, and a general vector in $\H_v$ has an associated well-defined representation as an infinite sum.
The basis $\{\xi_{\kappa,\sigma,k}\}$ is similarly defined.
Note that, for $d=3$,  $\rho_v(g)$ has multiplicity three, associated with the decomposition
\[
\H_v = \prod_{\kappa\in F_v, k\in \bZ^d} \bR \oplus \bR \oplus \bR,
\]
whereas $\rho_e(g)$ has multiplicity one.

We now define  a representation $\tilde{\rho_v}$ of the space group $\G(\C)$
as affine maps of  $\H_v$. If $T$ is an  isometry of $\bR^d$
let $\tilde{T}$ be the
 affine map on $\H_v$ determined by coordinate-wise action of $T$ and let
\[
\tilde{\rho}_v(g):=\rho_v(g)\tilde{T_g}=\tilde{T_g}\rho_v(g)
\]
where $g \to T_g$ is the affine isometry representation $\rho_{\rm sp}$.
One may also identify $\tilde{\rho}_v(\cdot)$  as the tensor product
representation $\rho_n(\cdot )\otimes \rho_{\rm sp}(\cdot )$ where $\rho_n(\cdot)$ is the multiplicity one version of $\rho_v(\cdot )$ (where the subscript "$n$" stands for "node").
Here the transformation $\rho_n(g)$ is linear while $\rho_{\rm sp}(g)$ may be affine.

In the next lemma and the ensuing discussion we make explicit the action of the transformation $\tilde{\rho}_v(g)$ on the space of affinely periodic velocity vectors. 

\begin{lem}\label{l:FlexCompare}
 Let $\C$ be a crystal framework with motif $(F_v,F_e)$ and
let $g$ be  a semiseparable element of the crystallographic space group $\G(\C)$. Let $T_g=TB$ be the
translation-linear factorisation
of the isometry $T_g=\rho_{\rm sp}(g)$,
where $Tw = w-\gamma$ for some $\gamma \in \bR^d$.
Also let  $\tilde{u}$ be a velocity vector in $\H_v^{\rm aff}$ which corresponds to the vertex-matrix pair $(u,A)$ (as  in Definition \ref{d:Flexes} and Theorem \ref{t:FlexCompare}). Then
$\tilde{\rho}_v(g)\tilde{u}$ is
a velocity vector
$\tilde{v}$ in $\H_v^{\rm aff}$ corresponding to
a pair $(v, BAB^{-1})$.
\end{lem}

\begin{proof}
Let $(\kappa',k')=g\cdot(\kappa,k)$.
Then
\[
(\tilde{\rho}(g)\tilde{u})_{\kappa',k'}=T_g\tilde{u}_{g^{-1}\cdot(\kappa',k')} =
T_g\tilde{u}_{\kappa,k} = T_g(u_\kappa-AZk)
\]
\[
=B(u_\kappa-AZk)-\gamma = w_1 -BAZk
\]
with $w_1=Bu_\kappa - \gamma = T_gu_\kappa$.

On the other hand $k$ may be related to $k'$. We have
\[
p_{\kappa',0}+Zk' = p_{\kappa',k'}= T_gp_{\kappa,k}
=T_g(p_{\kappa,0}+Zk)
\]
and so, since   $T_g^{-1}= B^{-1}T^{-1}$,
\[
Zk= T_g^{-1}(Zk'+p_{\kappa',0})-p_{\kappa,0}
=B^{-1}(Zk'+p_{\kappa',0}+\gamma)-p_{\kappa,0}.
\]
Thus $Zk$  has the form $B^{-1}Zk'+w_2$ where $w_2=T_g^{-1}p_{\kappa',0}-p_{\kappa,0}$  is independent of $k'$.
Substituting,
\[
(\tilde{\rho}(g)\tilde{u})_{\kappa',k'}= w_1-BAZk = (w_1 - BAw_2) -BAB^{-1}Zk'.
\]
Also, since $g$ is semiseparable the vector $v= w_1 - BAw_2$ depends only on $\kappa'$, and so the proof is complete.
\end{proof}

It follows from the lemma that if all symmetries are semiseparable for $F_v$ then $\tilde{\rho}_v$ determines a finite-dimensional representation
$\pi_v$ of $\G(\C)$ which is given by restriction to $\H_v^{\rm aff}$. We write this as $\pi_v(g)= \tilde{\rho}_v(g)|_{\H_v^{\rm aff}}$.
The proof of the lemma also provides detail for a coordinatisation of
this representation and 
we now give the details of this.

We first introduce what might be referred to to as the unit cell representation of $\G(\C)$ (or a subgroup of  semiseparable symmetries).  This is the finite-dimensional representation $\mu_v$ on $\bR^{d|F_v|}$ given by
\[
(\mu_v(g)u)_{\kappa} = T_gu_{g^{-1}\cdot\kappa}, \quad \kappa \in F_v.
\]
That this is a representation follows from  the observation that it is
identifiable as a tensor product representation
\[
\mu_v: g \to  \nu_n(g) \otimes \rho_{\rm sp}(g)
\]
where $\nu_n$ is the vertex class permutation representation on $\bR^{|F_v|}$, so that
$\mu_v =   \nu_n \otimes Id_3$.
This representation features as a subrepresentation of $\pi_v$ as we see below.

There is a companion edge class representation for the range space $\H_e$. Define first
the linear transformation representation $\rho_e$ of $\G(\C)$ on $\H_e$ given  by
\[
(\rho_e(g)w)_{f,k}=w_{g^{-1}\cdot(f,k)}, \quad \quad f\in F_e, k\in \bZ^d.
\]
This in turn provides a finite-dimensional representation
$\pi_e$ of $\G(\C)$ on the finite-dimensional space $\H_e^{per}$ of \textit{periodic} vectors, that is, the
space of $\rho_e(l)$-periodic vectors for $l\in \bZ^d$.
These periodic vectors $w$ are determined by the (scalar) values $w_{f,0}$, for $f\in F_e$.
Writing $\eta_f, f \in F_e,$ for the natural basis elements for $\bR^{|F_e|}$ we have the 
edge class permutation matrix representation
\[
\pi_e(g)\eta_f = \eta_{g^{-1}\cdot f},\quad f\in F_e, \quad g \in \G(\C).
\]

In the formalism preceding
Theorem \ref{t:FlexCompare}, we have identified $\H^{\rm aff}_v$ with a finite-dimensional
vector space which we now denote as  $\D$. This is the domain of $R(\M,\bR^{d^2})$;
\[
\D:=\bR^{d|F_v|}\oplus \bR^{d^2} = (\bR^{|F_v|}\otimes \bR^d)\oplus (\bR^d \oplus \cdots \oplus \bR^d).
\]
Let $D:\H^{\rm aff}_v \to \D$ be the  linear identification in which the {vector} $D\tilde{u}$  corresponds to the matrix-data form $(u,A)$ written as a row vector of columns.
Thus
\[
\pi_v(\cdot)=D\tilde{\rho}_v(\cdot)|_{\H^{\rm aff}_v}D^{-1}.
\]

Note that
any transformation in $\L(\D)$, the space of all linear transformations on $\D$,
has a  natural $2 \times 2$ block-matrix representation.
With respect to this we have the block form
\[
\pi_v(g) = \begin{bmatrix}
\mu_v(g)& \Phi_1(g) \\0& \Phi_2(g)\end{bmatrix}, \quad g \in \G(\C),
\]
where $\Phi_2(g)(A) = BAB^{-1}$.
To see this, return to the proof of Lemma
\ref{l:FlexCompare} and note that in the correspondence
\[
 {\tilde{\rho}_v(g)}: (u,A) \to (v,BAB^{-1}),
\]
if $A =0$ then  $v_{\kappa'}=T_gu_\kappa$ and so $v = \mu_v(g)u$ in this case.

To see the nature of the linear transformation $\Phi_1$ note that from the  lemma
that
\[
(\Phi_1(g)(A))_{\kappa'} = (-BAw_2)_{\kappa'} = -BA(T_g^{-1}p_{\kappa',0}- p_{\kappa,0}).
\]
In particular if $g$ is a fully separable symmetry for the periodicity, so that $T_gp_{\kappa,0}$ has the form
$p_{\kappa',0}$ for all $\kappa$, then $\Phi_2$ is the zero map and  a simple block diagonal form holds for $\pi_v(g)$, namely,
\[
\pi_v(g)(u,A) = (\mu_v(g)u, T_gAT_g^{-1}).
\]
In fact it follows that we also have this block diagonal form if $g$ is a symmetry such that $T_g$ acts as a permutation of the motif set $F_v =\{p_{\kappa, 0}\}$. Such a symmetry is necessarily of finite order.
When this is the case we shall say that the symmetry \emph{$g$ acts on $F_v$}.
(See Corollary \ref{c:FG}, Theorem \ref{t:MC} and Corollary \ref{c:IneqPredictor}.)

\subsection{Symmetry equations}

The following theorem is a periodic framework version of the symmetry equation given in
Owen and Power \cite{owe-pow-FSR}.

\begin{thm}\label{t:SymmetryEquation1}
Let $\C$ be a crystal framework in $\bR^d$ and let $ g \to T_g$ be a representation
of the space group $\G(\C)$ as isometries of $\bR^d$.

(a) For the  rigidity matrix transformation
$R(\C):\H_v \to \H_e$,
\[
\rho_e(g)R(\C)=R(\C)\tilde{\rho}_v(g), \quad g\in \G(\C).
\]

(b) Let $\G\subseteq \G(\C)$ be a subgroup of semiseparable symmetries. Then,
for the affinely periodic rigidity matrix $R(\M,\bR^{d^2})$ determined by
the motif $\M=(F_v, F_e)$,
\[
\pi_e(g)R(\M,\bR^{d^2})=R(\M,\bR^{d^2})\pi_v(g), \quad g \in \G,
\]
where $\pi_e$ is the edge class representation
of $\G$ in $\bR^{|F_e|}$ induced by $\rho_e$ and where
$\pi_v$ is the  representation of $\G$ in
$(\bR^{|F_v|}\otimes \bR^d) \oplus \bR^{d^2}$
induced by $\tilde{\rho}_v$.
\end{thm}

\begin{proof}By Theorem \ref{t:FlexCompare} and Lemma \ref{l:FlexCompare} we have an identification
of the space $\H^{\rm aff}_v$ of affinely periodic velocity vectors with the domain of $R(\M,\bR^{d^2})$. It follows from the preceding discussion  that the symmetry equation given in (b)  follows from the symmetry equation in (a).

To verify (a) observe first  that
if $(G, p)$ is the framework for $\C$ with labelling as given in Definition \ref{d:InfiniteRigidityMatrix} then the  framework for $(G, g\cdot p)$ with relabelled framework vector $g\cdot p$ with $(g\cdot p)_{\kappa,k} = p_{g\cdot(\kappa,k)}$  has rigidity matrix
$R(G,g\cdot p)$. Examining matrix entries shows that this matrix  is equal to the row and column permuted matrix $\rho_e(g)^{-1}R(G,p)\rho_v(g)$.

On the other hand  the row for the edge $e=[p_{\kappa,k},p_{\tau,l}]$ has entries in the columns for $\kappa$ and for $\tau$ given by the coordinates of $T_gv_e$ and $-T_gv_e$ respectively, where $v_e=
p_{\kappa,k}-p_{\tau,l}$. If $T_g = TB$ is the translation-linear factorisation, with $B$ an orthogonal matrix, note that
\[
T_gv_e=TBp_{\kappa,k} -TBp_{\tau,l} = Bv_e.
\]
Here $v_e$ is a column vector, while the $\kappa$ column entries are given by the
row vector transpose $(Bv_e)^t =v_e^tB^t = v_e^tB^{-1}.$ Thus
\[
R(G, g\cdot p)= R(G,p)(I\otimes B^{-1}).
\]
Note also, that for any rigidity matrix $R$ and any affine translation $S$ we have
$R(I \otimes S) = R$, as  transformations from $\H_v$ to $\H_e$. Thus we also have
\[
R(G,g\cdot p)=  R(G,g\cdot p)(I\otimes T^{-1})=   R(G,p)(I\otimes B^{-1})(I\otimes T^{-1})
\]
which is $R(G,p)(I\otimes T_g^{-1})$
and (a) follows.
\end{proof}
\bigskip

As an application we obtain general  Fowler-Guest style Maxwell Calladine formulae relating the traces (characters) of symmetries $g$ in various subrepresentations of $ \pi_v,  \pi_e$ derived from flexes and stresses.

Let $\E$ be a fixed subspace  of $d \times d$ axial velocity matrices $A$, let  $ g$
be a semiseparable symmetry for which
$T_gA = AT_g$ for all $A\in \E$, and let $H$   be the cyclic subgroup generated by $g$. It follows from Lemma \ref{l:FlexCompare}
that there is a restriction representation $\pi^\E_v$ of $H$ on the space $\H^\E_v:=\bR^{d|F_v|} \oplus \E$ and that there is an associated intertwining symmetry equation.
This equation implies that the space of $\E$-affinely periodic infinitesimal flexes is an invariant subspace for the representation $\pi_v$ of $H$. Thus there is an associated restriction representation of  $H$, namely $\pi_{\rm fl}^\E$.
We note the following subrepresentations of $\pi_v^\E$:
\medskip

$\pi_{\rm fl}^\E$
on the invariant subspace
$\H^\E_{\rm fl} := \ker R(\M,\E)$,
\medskip

$\pi_{\rm rig}^\E$ on the invariant subspace $\H^\E_{\rm rig}:=\H^\E_v \cap \H_{\rm rig}$,
\medskip

$\pi_{\rm mech}^\E$ on the invariant subspace $\H^\E_{\rm mech}:=\H^\E_{\rm fl}\ominus \H^\E_{\rm rig}$.
\medskip

\noindent Also, we have a  subrepresentation of $\pi_e^{\rm per}$, namely

\medskip

$\pi_{\rm str}^\E$ on the invariant subspace $\H^{\E}_{\rm str} := \coker R(\M,\E)$.
\medskip

\begin{cor}\label{c:FG}
Let $\C$ be a crystal framework and let $g$ be a semiseparable space group symmetry or a symmetry which acts on $F_v$. Also, let $\E\subseteq \bR^{d^2}$ be a space of affine velocity matrices which commute with $T_g$.
Then
\[
tr(\pi_{\rm mech}^\E(g)) -tr(\pi_{\rm str}^\E( g))
= tr({\pi_{v}}^\E( g))-tr(\pi_{e}( g))
-tr(\pi_{\rm rig}^\E( g)).
\]
\end{cor}

\begin{proof}
The rigidity matrix $R(\M, \E)$ effects an equivalence between the subrepresentations of $H$ on the spaces
$\H^\E_v\ominus \H^\E_{\rm fl}$ and $\H_e\ominus \H_{\rm str}^\E$. Thus the traces
of the representations of $g$ on these spaces are equal and from this the identity follows.
Indeed, with respect to the decompositions
\[
\H^\E_v= (\H^\E_v\ominus \H^\E_{\rm fl})\oplus \H^\E_{\rm mech} \oplus \H^\E_{\rm rig}
\]
\[
\H^\E_e= (\H_e\ominus \H^\E_{\rm str})\oplus \H^\E_{\rm str}
\]
the rigidity matrix has block form
\[
R(\M,\E)=\begin{bmatrix}
R_1&0&0\\0& 0&0\end{bmatrix}
\]
with $R_1$ an invertible square matrix. Thus $R_1$ effects an equivalence between the representations on the first
summands above. It follows that the traces of these representations of $g$ are equal and so
\[
tr(\pi_v^\E(g))-tr(\pi_{\rm mech}^\E(g))-tr(\pi_{\rm rig}^\E( g))=tr({\pi_{e}}( g))-tr(\pi_{\rm str}( g))
\]
as required.
\end{proof}

Recall that the \emph{point group}  of a crystal framework is the quotient  group  $\G_{\rm pt}(\C)=\G(\C)/ \bT$ 
determined by a maximal 
subgroup $\T=\{T_l:l\in \bZ^d\}$.
In the next theorem we suppose  that the space group is separable, so that
$\G(\C)$ is isomorphic to $\bZ^d\times \G_{\rm pt}(\C)$.
Write  $\dot g$ for the coset element in $\G_{\rm pt}(\C)$ of
an element $g$ of the space group.
Then, following an appropriate shift, one can recoordinatise the bar-joint framwork $\C$
so that for the natural inclusion map $i: \G_{\rm pt}(\C)\to \G(\C)$  the map $\dot g \to T_{i(\dot g)}$ is a representation by \textit{linear} isometries.
In this case we obtain from the restriction of $\pi_v$ and $\pi_e$ to $\G_{\rm pt}(\C)$ the representations
\[
\dot \pi_v : \G_{\rm pt}(\C) \to \L(\D),\quad \dot \pi_e : \G_{\rm pt}(\C) \to \L(\bR^{|F_e|}).
\]
In fact $\dot \pi_v $ has a block diagonal form with respect to the direct sum decomposition $\D= \bR^{d|F_v|}\oplus \bR^{d^2}$ and is well-defined by the recipe
\[
\dot \pi_v(\dot g) = \begin{bmatrix}
\mu_v(g)& 0 \\0& \Phi_2(g)\end{bmatrix}, \quad g \in \dot g \in  \G_{\rm pt}(\C),
\]
where $\Phi_2(g):A \to T_gAT_g^{-1}$.
Also  $\dot \pi_e$ is similarly explicit and agrees with the
natural edge class permutation representation $\nu_e$ which is well-defined by
\[
\nu_e(\dot g)\eta_f =\eta_{g^{-1}\cdot f}, \quad f \in F_e, \quad g \in \dot g \in \G_{\rm pt}(\C).
\]

\begin{thm} \label{t:SymmetryEquation2a}For a crystal framework with separable space group $\G(\C)$,
\[
\dot \pi_e(\dot g)R(\M,\bR^{d^2})=R(\M,\bR^{d^2})\dot \pi_v(\dot g),\quad  \dot g \in \G_{\rm pt}(\C).
\]
Moreover for a fixed element $\dot g \in \G_{\rm pt}(\C)$ the {individual} symmetry equation
\[
\dot \pi_e(\dot g)R(\M,\E)=R(\M,\E)\dot \pi_v(\dot g)
\]
holds where $\E$ is any space of matrices  which is invariant
under the map $A \to  T_{\dot g }AT_{\dot g }^{-1}$.
\end{thm}

\begin{proof}
The equations follow from those of
Theorem \ref{t:SymmetryEquation1}.
\end{proof}

\subsection{Symmetry-adapted Maxwell Calladine equations.}
We now give  Maxwell-Calladine formulae for symmetry respecting mechanisms and self-stresses and which
which  derive from consideration of velocity vector spaces in which \textit{all} vectors
are fixed under an individual symmetry.

Suppose first that $g$ is a space group element in $\G(\C)$ which is of separable type or which acts on $F_v$.
Then $T_g$ is an orthogonal linear transformation and  $\pi_v(g)$ takes the
block diagonal form
\[
\pi_v(g)(u,A) = (\mu_v(g)u, T_gAT_g^{-1}).
\]

Let $\H^g_v$ be the subspace of vectors
$(v,A)$ which are fixed by the linear transformation $ \pi_v(g)$. Then $\H^g_v$ splits as a direct sum which we write as
$
\F_g\oplus \E_g.
$
where  $\E_g$ is the space of matrices commuting with $T_g$ and
$\F_g$ is the space of vectors $v$ that are fixed by the linear transformation $\mu_v(g)=\nu_n(g)\otimes T_g$.

Let $\H^g_{e}$ be the subspace of vectors in $\bR^{|F_e|}$ that are fixed by $\pi_e(g)$, so that
$\H^g_{e}$ is naturally identifiable with $\bR^{e_g}$ where $e_g$
is the number of orbits of edges in $F_e$ induced by $g$.

It follows from the symmetry equations in Theorem \ref{t:SymmetryEquation2a}  that the rigidity matrix transformation $R(\M,\E_{g})$
maps $\H^g_v$ to $\H^g_{e}$.

In the domain space write
$\H^g_{\rm rig}$ for the space of ($  \pi_v(g)$-invariant) rigid motions in $\H_v^g \subseteq \bR^{d|F_v|}\oplus \E_g$.
Similarly write $\H^g_{\rm fl}$ for the space of $   \pi_v(g)$-invariant affinely periodic infinitesimal flexes $(u, A)$  in $\H_v^{g}$ and   write $\H^g_{\rm mech}$ for the orthogonal complement space
$\H^g_{\rm fl}\ominus \H^g_{\rm rig}$.

In the co-domain space $\H^g_e$ of periodic $\pi_e(g)$-invariant vectors we simply have the subspace
$\H^g_{\rm str}$ of $ \pi_e(g)$-invariant periodic infinitesimal self-stresses.

Finally let
\[
m_g=\dim \H^g_{\rm mech}, \quad s_g= \dim \H^g_{\rm str},\quad  e_g =\dim \H^g_e, \quad f_g=\H^g_{\rm rig}.
\]
Noting that $\dim \H_v^g = \dim \F_g +\dim \E_g$
the following symmetry-adapted Maxwell-Calladine formula follows from Corollary \ref{c:FG}.

\begin{thm}\label{t:MC}
Let $\C$ be a crystal framework and let $g$ be a separable symmetry of  $\G(\C)$ or a symmetry which acts on $F_v$. Then
\[
m_g - s_g = \dim \F_g + \dim\E_g- e_g - f_g.
\]
\end{thm}

The right hand side of this equation is readily computable as follows:
\medskip

(i)  $\dim \F_g$ is the dimension of the space of vectors in $\bR^{|F_v|}\otimes \bR^d$
which are fixed by the linear transformation $\nu_n(g)\otimes T_g$.
\medskip

(ii) $\dim \E_g$ is the dimension of the commutant of $T_g$, that is of the linear space of $d\times d$ matrices $A$ with $AT_g=T_gA$.
\medskip

(iii) $e_g$ is  the number of orbits in the motif set $F_e$ under the action of $g$.
\medskip

(iv) $f_g$ is the dimension of the subspace of infinitesimal rigid motions in $\F_g$.
\medskip

For a general symmetry $g$ in $\G(\C)$, such as a glide reflection, the above applies except for the splitting of the space $\H_v^g$ of $\pi(g)$-fixed vectors. In this case we obtain the Maxwell-Calladine formula
\[
m_g - s_g = \dim \H_v^g - e_g - f_g
\]
and $\dim \H_v^g$ is computed by appeal to the full block triangular representation of the
transformation $\pi_v(g)$.

For strictly periodic flexes rather than affinely periodic flexes, that is, for the case $\E=\{0\}$,
there is a modified formula with  replacement of
$f_g$  by the dimension, $f_g^{per}$ say, of the subspace $\H_v^{per}$ of $g$-symmetric periodic
velocity vectors corresponding to rigid motions. Thus for the identity symmetry $g$ we obtain
\[
m-s=d|V_f|-|V_e|-d
\]
as already observed in Theorem \ref{t:M-Caffine}.

We remark that the restricted transformation $R(\M,\E_g):\H_v^g \to \H_e^g$ is a coordinate free form
of the orbit rigidity matrix  used by Ross, Schulze and Whiteley \cite{ros-sch-whi} (see also Schulze and Whiteley \cite{sch-whi}). With this matrix they derive
counting inequalities as predictors for infinitesimal mechanisms and we give a general such formula in the next corollary.
The analysis in \cite{ros-sch-whi} also indicates how finite motions (rather than infinitesimal ones) arise if, for example,
the framework is sufficiently generic apart from the demands
of a specific symmetry and periodicity.
\medskip

\begin{cor}\label{c:IneqPredictor}
Let $\C$ be a crystal framework and let $g$ be a separable symmetry of  $\G(\C)$ or a symmetry which acts on $F_v$.
If
\[
e_g < \dim \F_g +\dim\E_g- f_g
\]
then there exists a non-rigid motion infinitesimal flex which is $T_g$-symmetric.
\end{cor}

\begin{proof}
Immediate from Theorem \ref{t:MC}.
\end{proof}

Finally we remark that if $(G,p)$ is an infinite bar-joint framework in $\bR^3$ which is generated by the translates of a finite motif under a translation \textit{subgroup} $\{T_k : k\in \bZ \subseteq \bZ^3\}$
then one can similarly obtain symmetry equations for the rigidity matrices $R(G,p)$.
One such example is the hexahedron tower illustrated in the next section.


\section{Some crystal frameworks.}

We now examine some illustrative  frameworks and note in particular that the various Maxwell-Calladine equations give readily computable information.
This is a potentially useful tool given the formidable size of the rigidity matrix for even moderate crystal frameworks.

A motif $\M=(F_v, F_e)$ for  $\C$, with specified period vectors, is a geometrical construct that  carries all the ingredients necessary for determining
of various rigidity matrices. It may be specifed in  more combinatorial terms
as a \textit{labelled motif graph}, by which we mean
the simple  abstract graph $(V(F_e), F_e)$ (with $V(E_f)$ the set of \textit{all} vertices from $F_e$)
together with a vertex labelling given by the $(\kappa, k)$ labelling and a distance
specification $d:F_e \to \bR_+$ given by the lengths of the edges. In our discussions,
in contrast with \cite{bor-str-2}, \cite{mal-the},  \cite{ros}, we have  not been concerned  with combinatorial characterisations, nor with realizability, and so we have had no need for this abstract form. However we note that
there are equivalences between such  doubly labelled graphs and the formalism of other authors, who incorporate {gain graphs} (Ross, Schulze, Whiteley \cite{ros-sch-whi}),
coloured graphs (Malestein and Theran \cite{mal-the}), and quotients of periodic graphs
(Borcea and Streinu \cite{bor-str}). For example the gain graph arises from
the labelled motif graph $(V(F_e), F_e)$ by merging vertices with the same $\kappa$ label
and assigning to the new edges a direction to one of the vertices, coming from the ancestral vertices $(\kappa, k)$, $(\tau, k')$, say $\kappa$, and assigning to this edge a label, namely $k-k'\in \bZ^d$ for the choice of $\kappa$ (and $k'-k$ for the choice of $\tau$).

\subsection{The kagome framework $\C_{\rm kag}$.} The kagome framework has been examined in a variety of engineering contexts. See \cite{hut-fle} for example. It also features  as slices in higher dimensional crystal frameworks such as the kagome net framework of $\beta$--cristobalite. See, for example, \cite{dov-exotic}, \cite{pow-matrix}, \cite{weg}. Its hexagonal structure
is implied by Figure \ref{f:kag} which also indicates a motif
\[
\M=(F_v, F_e), \quad F_v=\{p_1, p_2, p_3\}, \quad F_e=\{e_1,\dots ,e_6\}.
\]
where $e_1=[p_{1,(0,0)},p_{2,(0,0)}], e_2 = [p_{2,(0,0)},p_{3,(0,0)}], e_3=[p_{3,(0,0)},p_{1,(0,0)}], \\
e_4=[p_{1,(0,0)},p_{2,(-1,0)}], e_5=[p_{2,(0,0)},p_{3,(1,-1)}], e_6=
[p_{3,(0,0)},p_{1,(0,1)}].
$

Consider a scaling in which the period vectors are
\[
a_1=(1,0),\quad a_2=(1/2, \sqrt{3}/2),
\]
so that the equilateral triangles of the framework have side length $1/2$.
\begin{center}
\begin{figure}[h]
\centering
\includegraphics[width=5.5cm]{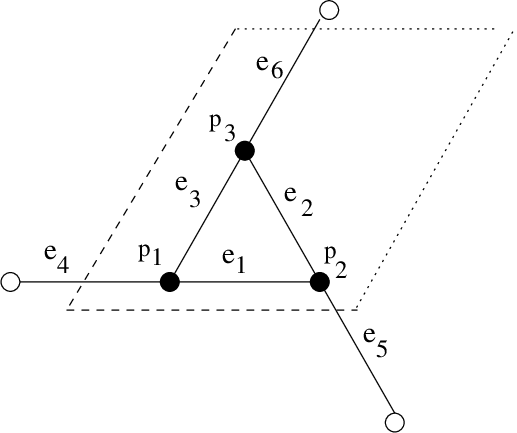}
\caption{A motif for the kagome framework, $\C_{\rm kag}$.}
\label{f:kag}
\end{figure}
\end{center}
The motif rigidity matrix $R(\M)$ is $6 \times 6$ while $R(\M,\bR^{d^2})$ is $6 \times 10$.
Examining the local condition for a self-stress (or examining  the co-kernel of $R(\M)$) one soon finds a basis of three vectors in $\bR^{|F_e|}$ for the space of periodic self-stresses. In fact it is natural to take the three vectors which assign coefficients of $1$ to a colinear pair of edges of $F_e$ and the coefficient $0$ to the other four edges.

For strictly periodic  vertex displacement velocities and flexes, that is, for the case $\E=\{0\}$,
the Maxwell-Calladine equation  of Theorem \ref{t:M-Caffine}
takes the form
\[
m - s = 2|F_v| + 0 - |F_e| - 2 =-2.
\]
Note in particular that for $\E=\{0\}$ there is a $2$-dimensional space of periodic infinitesimal flexes of rigid motion type and so $f_\E=2$.
Since $s=3$ we have $m=1$. This infinitesimal mechanism corresponds to the unique (up to scalar) "alternating rotation mechanism",
with the motif triangle infinitesimally rotating about its centre. We may write this in matrix data form as $(u_{\rm rot},0)$

To compute the motif rigidity matrix $R(\M)$ and the affine rigidity matrix
$$R(\M, \bR^{d^2})=[R(\M)\quad X(\M)]$$
we note that
\[
v_{e_1} = (-1/2,0), \quad v_{e_2} = (1/4,-\sqrt{3}/4), \quad v_{e_3} = (1/4,\sqrt{3}/4),
\]
\[
v_{e_4} = (1/2,0), \quad v_{e_5} = (-1/4,\sqrt{3}/4), \quad v_{e_6} = (-1/4,-\sqrt{3}/4),
\]
\[
\delta(e_4) = (-1,0), \quad \delta(e_5) = (1,-1), \quad \delta(e_6) = (0,1).
\]
This leads to the identification of  $R(\M, \bR^{9})$ as the matrix
{\tiny
\[
\left[ \begin {array}{cccccccccc} -1/2&0&1/2&0&0&0&0&0&0&0\\
\noalign{\medskip}0&0&1/4&-\sqrt{3}/4&-1/4&\sqrt {3}/4&0&0&0&0\\
\noalign{\medskip}-1/4&-\,\sqrt {3}/4&0&0&1/4&\sqrt{3}/4&0&0&0&0\\
\noalign{\medskip}1/2&0&-1/2&0&0&0&-1/2&0&0&0\\
\noalign{\medskip}0&0&-1/4&\sqrt {3}/4&1/4&-\sqrt{3}/4&-1/4&\sqrt{3}/4&1/4&-\sqrt{3}/4\\
\noalign{\medskip}{1/4}&\sqrt{3}/4&0&0&-1/4&-\sqrt {3}/4&0&0&-1/4&-\sqrt{3}/4\end {array}
\right].
\]
}

In the affine case, with $\E=M_2(\bR)$, we now have $f_\E=3$, since infinitesimal rotation is now admissible,
and so
\[
m_\E-s_\E=d|F_v|+\dim \E -|F_e| - f_\E = 6+4-6-3= 1.
\]
From the affine rigidity matrix it follows readily that there are no linear dependencies among the rows
and so $s_\E=0$. (Roughly speaking, no periodic self-stresses remain on admitting all affine motions.) Thus $m_\E=1$ indicating that affine freedom confers no additional (linearly independent) affinely periodic infinitesimal flex beyond infinitesimal rotation. One can check in particular that while there is "purely affine" flex fixing the motif vertices,
with matrix data $(0,A)$),  this flex is a linear combination of rigid rotation
and the alternation rotation  $(u_{\rm rot},0)$.


Consider the rotation symmetry $g$ whose isometry $T_g$ is
rotation by $\pi/3$ about the centre of the triangle in the motif.
This symmetry acts on $F_v$.
The space $\E_g \subseteq M_d(\bR)$, of the
symmetry-adapted Maxwell equation in Theorem \ref{t:MC},
has dimension $2$.
Also $f_g=1$ since now only infinitesimal rotation
is admissible as a rigid motion.
The space $\F_g$ is the space of fixed vectors for the linear transformation
$\mu_v(g)=\nu_n(g)\otimes T_g$ which has the form $S\otimes T_g$ where $S$ is a cyclic shift
of order $3$. It follows that $\dim \F_g$ is $2$.
Thus
\[
m_g - s_g = \dim \F_g+ \dim \E_g - e_g - f_g = 2+2-2-1=1.
\]

Since $m_g \leq 1$ it follows that   $s_g =0$.
While there is a $g$-symmetric self-stress for strict periodicity, this stress is invalidated by the admission of affine
motion and does not contribute to $s_g$.

\subsection{A Roman tile framework.}
A rather different hexagonal framework is the crystal  framework $\C_{\rm rom}$
based on the Roman tiling  indicated in Figure \ref{f:roman}.
The framework is rich in infinitesimal flexes. For example note that there are linear "channels" of alternating squares and hexagons which are joined by edges that are all orthogonal to an infinite line. There are thus nontrivial infinitesimal flexes that fix the vertices in one half-plane of such a line and infinitesimally translate all the other vertices in the line direction.
Another evident infinitesimal flex effects identical infinitesimal rotations on every triangular unit. This is a (strictly) periodic
infinitesimal mechanism with 6-fold rotational symmetry. We also remark that there are less evident \emph{local flexes} whose support (nonzero velocities) lie in a $2$-fold supercell of four parallelograms. Such a flex corresponds to a cycle of translational motions of the six triangular units of a ring of triangles. 

To determine the affine rigidity matrix note that there is a natural parallelogram unit cell in Figure \ref{f:roman} which encloses $6$ framework vertices. We take this to be the motif set $F_v$ and to determine the translational symmetry group $\T$. The $8$
internal edges of this unit cell together with an appropriate choice of $4$ cell-spanning edges
provides a possible set $F_e$. We may choose these edges to be the two westward edges and the two south-east edges. The resulting motif $\M$ determines
a  $12 \times 16$ affinely periodic rigidity matrix $R(\M, \bR^{d^2})$. This in turn determines the affinely periodic infinitesimal flexes and the associated self-stresses.

\begin{center}
\begin{figure}[h]
\centering
\includegraphics[width=7cm]{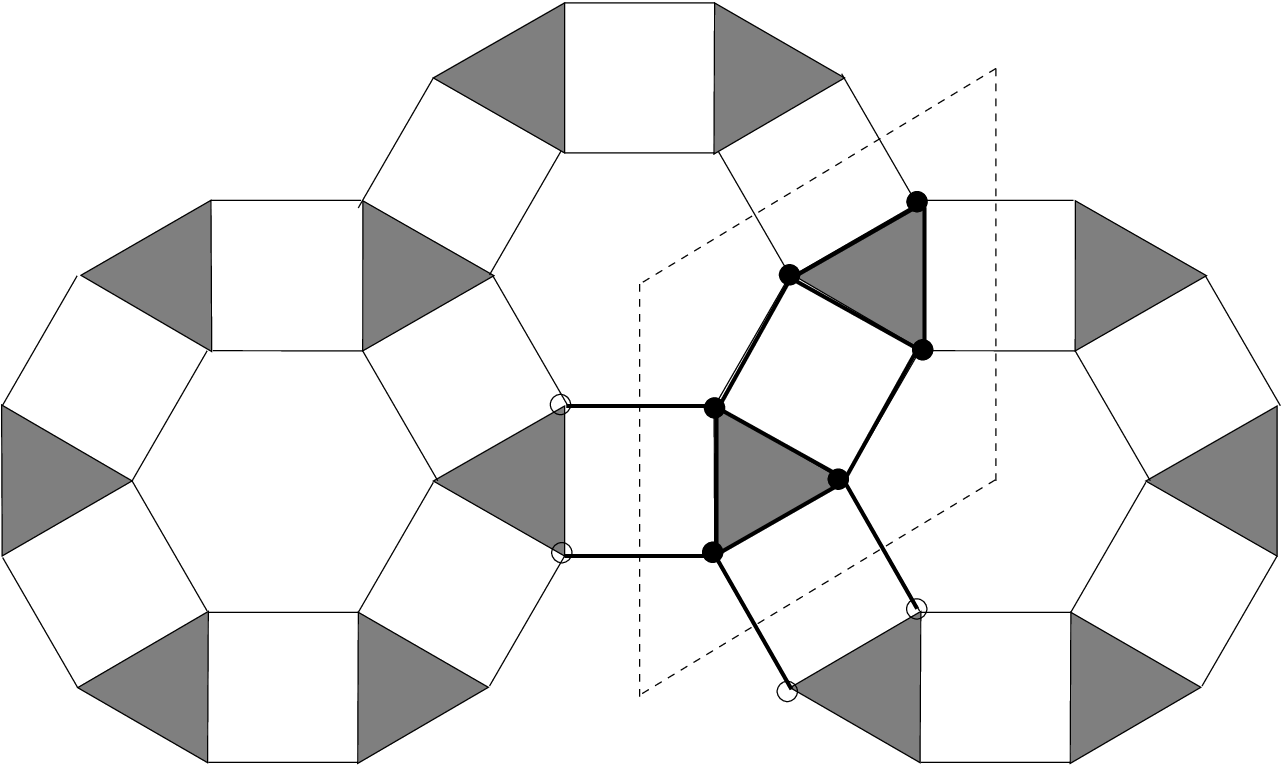}
\caption{Part of the Roman tiling framework $\C_{\rm rom}$.}
\label{f:roman}
\end{figure}
\end{center}

From Theorem \ref{t:M-Caffine} the dimension $m_{per}$  of the space of strictly periodic mechanisms ($\E=\{0\}$) satisfies
$$
m_{per}-s= d|F_v|+\dim \E-|F_e|-2= 12+0-12-2 =-2
$$
In view of the triangle rotation flex $m_{per}\geq 1$ and so there are at least $3$ linearly independent periodic self-stresses.

On the other hand
the unrestricted affine variant of the Maxwell-Calladine equation
gives
\[
m_{\rm aff} - s_{\rm aff}= d|F_v| + d^2- |F_e| - f_{\rm rig} = 12+4-12-3 = 1,
\]
with  $f_{\rm rig}= 3$ since
there are now $3$ independent rigid motion flexes. It follows that there is at least a $1$-dimensional space of linearly independent affinely periodic infinitesimal mechanisms.

Let $g$ be the rotational symmetry of $\C_{\rm rom}$ which has period  $6$. This is semiseparable with respect to the motif indicated in Figure \ref{f:roman}, where $F_v$ consists of the $6$ framework points inside the unit cell parallelogram. However, the corresponding representation $\pi_v(\cdot)$ of $g$, and its generated cyclic group, (see Theorem \ref{t:SymmetryEquation1}) does not block-diagonalise. Accordingly we consider a motif $F_v'$, for the same translation group, on which $g$ acts. For this we may take $F_v'$ to be the six framework points of a hexagon subframework, with $F_e'$ the edges of the hexagon plus $6$ further edges being the $g$-orbit of an additional edge incident to a point in $F_v'$. Since $g$ acts on $F_v'$ the symmetry adapted Maxwell-Calladine formula applies.  In this case we have
\[
m_g - s_g = \dim \F_g+ \dim \E_g - e_g - f_g = 2+2-2-1=1.
\]
Indeed, $\mu_v(g)$ has the form $S\otimes T_g$, which is a cyclic shift on $\bR^{12}$ with a $2$-dimensional fixed point space, while $\E_g$, the space of matrices commuting with $T_g$  has dimension $2$. 
It follows directly from this calculation that that there is an affinely periodic infinitesimal mechanism with rotational symmetry.

\subsection{Rigid unit frameworks}
The planar diagram of Figure \ref{f:kite1} shows a template for a  translationally periodic
body-pin framework in the plane with kite-shaped rigid unit bodies. This framework is equivalent, from the point of view of rigidity, to a bar-joint framework 
formed by the edges of the quadrilaterals together with added internal edges. Figure \ref{f:kite1} shows a motif $(F_v, F_e)$ for a (nonmaximal) translation group such that an inversion symmetry $g$ acts on $F_v$. 
\begin{center}
\begin{figure}[h]
\centering
\includegraphics[width=7cm]{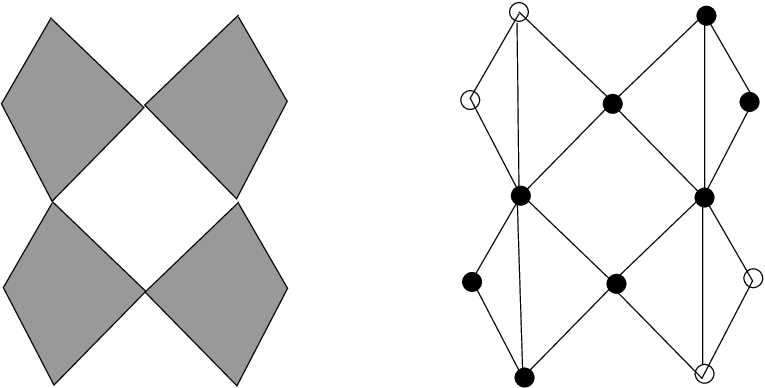}
\caption{Template and motif for a rigid unit framework.}
\label{f:kite1}
\end{figure}
\end{center}

That there is a nontrivial infinitesimal flex which is periodic  may be seen from the symmetry-adapted Maxwell-Calladine equation for  $g$.
The vertex motif vector space $\F_g$ is the sum of four $2$ dimensional spaces. The commutant space $\E_g$ for the linear inversion $T_g: (x,y)\to (-x,-y)$ is the set of all matrices and so has dimension $4$. The number $e_g$ of edge classes (including two classes for the added vertical edges) is $10$, while the number $f_g$  of independent inversion symmetric rigid  motions
is $1$, for rotation. It follows that
\[
m_g-s_g= \dim\F_g + \dim \E_g-e_g-f_g = 8+4-10-1 =1
\]
and so $m_g \geq 1$. 

By way of contrast note that the framework $\C_{\rm kite}$ implied by Figure \ref{f:kite2} has no inversion symmetry.
\begin{center}
\begin{figure}[h]
\centering
\includegraphics[width=5cm]{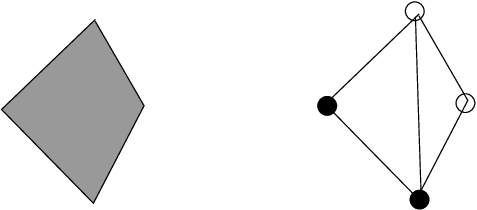}
\caption{The rigid unit for kite framework $\C_{\rm kite}$, and a minimal motif.}
\label{f:kite2}
\end{figure}
\end{center}
Although there are no nontrivial periodic flexes, for any translation group, there is an essentially unique affinely periodic infinitesimal flex for the
maximal translation group. This flex is given by alternating infinitesimal rotations of the rigid units about the midpoints of the vertical diagonal edges. The affine nature of this flex is expressed in the 
geometrically diminishing velocities in the direction of the positive $x$-axis. 

\subsection{The hexahedron framework $\C_{\rm Hex}$.}
A hexahedron, or triangular bipyramid,  with equilateral triangle faces,  may be formed by joining two tetrahedra at a common face. Let $\C_{\rm Hex}$ be the hexahedron crystal framework in which
bipyramids are vertically oriented and horizontally connected, in a regular triangular fashion, and where copies of
such horizontal layers are stacked on top of each other.
More precisely, $\C_{\rm Hex}$  is determined by a single bipyramid and the period vectors
\[
a_1=(1,0,0),\quad a_2=(1/2,\sqrt{3}/2,0)\quad a_3=(0,0,2h),
\]
where  $h= \sqrt{2}/\sqrt{3}$.
A motif $\M=(F_v, F_e)$ is given by the vertex set
$F_v=\{p_1, p_2\}$, with equatorial vertex $p_1=(0,0,0)$ and south polar vertex $p_2=(1/2,\sqrt{3}/6,-h)$, and the edge set $F_e$ consisting of the $9$  edges of the bipyramid.
Thus the polar vertices have degree $6$ while the equatorial vertices have degree $12$. 
The motif rigidity matrix
$R(\M)$ is a $9 \times 6$ matrix and the affine rigidity matrix
$R(\M, \bR^{9})$ is $9 \times 15$.
The fully affine version of the Maxwell-Calladine equation in Theorem \ref{t:M-Caffine}
for $\E=\bR^9$ gives
\[
m_\E-s_\E=d|F_v|+\dim \E -|F_e| - f_\E = 2\times 3 +9-9-6 =0.
\]
which is inconclusive on its own.
On the other hand,  the  "internal"  edge $[p_1, p_2]$ of the motif, with exponent $(0,0,0)$,
gives a row of zeros  to the associated row of the $9\times 9$ submatrix $X(\M)$ of the affine rigidity matrix. Thus $X(\M)$ has rank less than $9$ and so there is a nonzero vector $(0,AZ)$
in the kernel of $R(\M, \bR^{9})$. 
We conclude then that $m_\E\geq 1$ and $s_\E \geq 1$.
In fact $\C_{\rm Hex}$ has nontrivial infinitesimal flexes, analogous to the affine flex of the kite framework, in which the bipyramids undergo infinitesimal rotation about (parallel) lines through their centroids which are parallel to an equatorial edge.

\begin{center}
\begin{figure}[h]
\centering
\includegraphics[width=4cm]{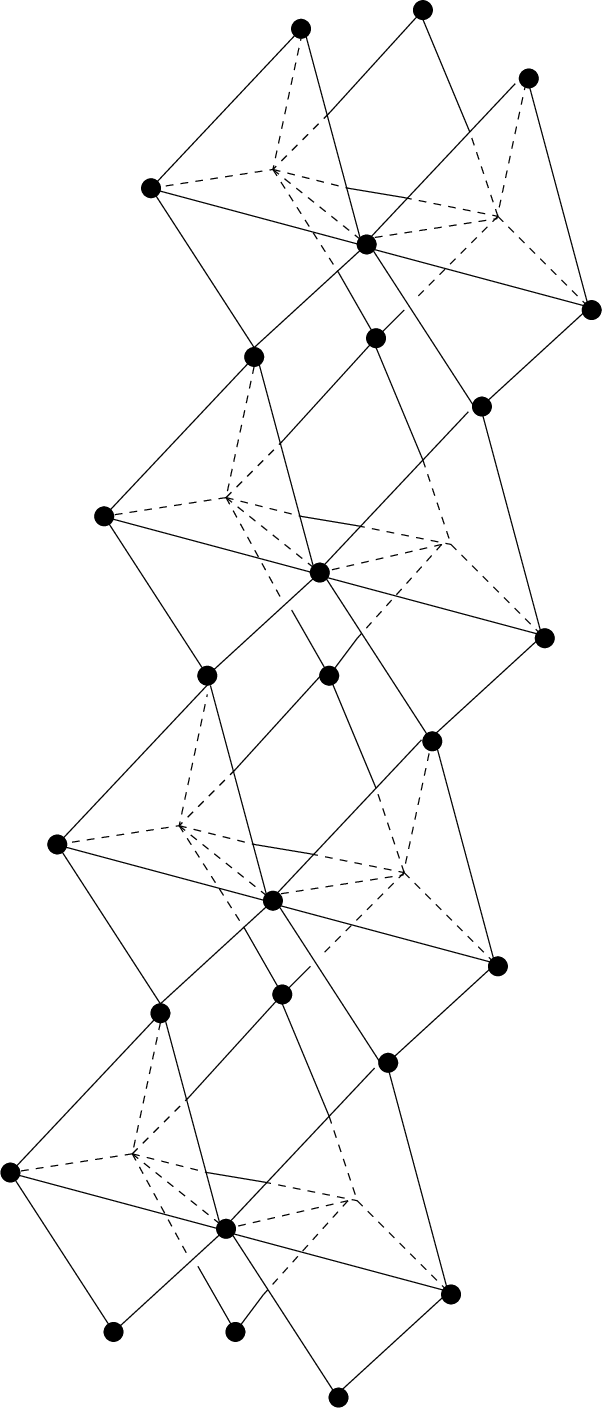}
\caption{Part of a hexahedron tower.}
\label{f:hextower}
\end{figure}
\end{center}

\subsection{The hexahedron tower}
We define
the {hexahedron tower framework} $\G_{\rm Hex}$ as the one-dimensionally translationally periodic framework in $\bR^3$ formed by stacking identical hexahedron $3$-rings infinitely in both directions, as indicated in Figure \ref{f:hextower}. In particular  $\G_{\rm Hex}$ is a subframework of  $\C_{\rm Hex}$.
A single hexahedron  $3$-ring
has  $3$ degrees of freedom beyond the $6$ isometric motion freedoms.
To see this note that if the connecting triangle is fixed, removing all continuous spatial motion possibilities, then  each of the three hexahedron units can flex independently. In view of the connections between the hexahedron $3$-rings it follows that any finite tower also has $3$ degrees of freedom in this sense.

The hexahedron tower has an evident rotationally symmetric
continuous flex $p(t)$ in which the triples of degree $6$ vertices  alternately  move a finite distance towards or away from the central axis. Viewing a single $3$-ring as the unit cell this flex is  flow-periodic with respect to $2$-cell periodicity and with respect to affine contraction along the central axis. The
associated infinitesimal flex $p'(0)$ is an affine ($2$-cell) periodic infinitesimal flex
of contractive affine type.
Such flexes are characterised by the property that the connecting vertex triples lie on planes which are parallel, as we see below.

For continuous flexes of the hexahedron tower we note  the following curious dichotomy.

\begin{prop}{Let $p(t), t\in[0,1],$ be a continuous flex of the hexahedron tower $\G_{\rm Hex}= (G_{\rm Hex},p)$ which fixes the connecting triangle of a single hexahedron $3$-ring. Then either $(G_{\rm Hex}, p(t))$ is  axially rotationally symmetric for all $t$ or
for some $s\in (0,1]$ the framework $(G_{\rm Hex}, p(s))$
is bounded.}
\end{prop}

\begin{proof} Assume first that
the framework
$(G_{\rm Hex}, p(s))$ has a hexahedron $3$-ring whose polar triples
determine planes which are not parallel.
The position of one polar triple in $\bR^3$ determines (uniquely) the position of the next triple while the subsequent triple is determined by reflection of the first triple in the plane through the second triple. From this reflection principle it follows that these planes and all subsequent planes, similarly identified, pass through a common line. Also it follows that $(G, p(s))$ is a bounded framework which circles around this line. (We are assuming here
that self-intersections are admissible.) If the angle between consecutive planes happens to be a rational multiple of $2\pi$ then the vertices of the deformed tower  occupy a finite number of positions with infinite multiplicity.

To complete the proof it suffices to show that if a single hexahedron $3$-ring $(H, p)$ is
 flexed to a position so that the plane of the three south pole vertices  $p_1', p_2', p_3'$ is parallel to the plane
of the three north pole vertices $p_4', p_5', p_6',$ then the $3$-ring is rotationally symmetric
(about the line through the centroid of the connecting triangle of equatorial vertices).

Consider the three-dimensional manifold of "flex positions" $(H, p')$ for which
$p_1'=p_1$ is fixed, $p_2'$ lies on a fixed open line segment between  $p_1$ and $p_2$, and $p_3'$ has two degrees of freedom in a small open disc in the fixed plane through $p_1, p_2, p_3$. The parallelism condition corresponds to the
two equations
\[
\langle p_4'-p_3', (p_2-p_1)\times(p_3-p_1) \rangle = 0, \quad
\langle p_5'-p_3', (p_2-p_1)\times(p_3-p_1) \rangle = 0.
\]
It follows that there is a one-dimensional submanifold of positions $p'$
with parallel triple-point planes. Since there is a one-dimensional submanifold
of rotationally  symmetric positions these positions provide all the positions $p'$, as required.
\end{proof}

By reversing the time parameter $t$ in the proof above, and assuming a decrease of the angle to zero at time $t=1$, one obtains a curious continuous "unwrapping flex", from a bounded "wrapped" hexahedron tower at time $t=0$ to the unbounded tower $\G_{\rm Hex}$ at time $t=1$, with bounded frameworks at  each intermediate time.

From the observations above we can conclude that $\C_{\rm Hex}$ is continuously rigid. Indeed, there is no continuous flex during which a hexahedron tower subframework becomes a bounded framework. (This leads to a violation of the triangle inequality in neighbouring towers.)   The only other possibility is that all hexahedron tower subframeworks undergo an axially symmetric continuous flex. However, because of their connectivity this is not possible.
\medskip

\emph{Added in proof.} We note that finite and infinite column frameworks are also considered in a recent paper of Whiteley
\cite{whi-fragments}.

\bibliographystyle{abbrv}
\def\lfhook#1{\setbox0=\hbox{#1}{\ooalign{\hidewidth
  \lower1.5ex\hbox{'}\hidewidth\crcr\unhbox0}}}

\end{document}